\documentclass[a4paper, 11pt]{article}
\usepackage{graphicx}
\usepackage{amsmath, amsthm, amssymb, amsfonts}

\usepackage[comma]{natbib}

\usepackage[colorlinks=true,breaklinks=true,bookmarks=true,urlcolor=blue,
     citecolor=blue,linkcolor=blue,bookmarksopen=false,draft=false]{hyperref}

\usepackage{graphicx}

\newtheorem{theorem}{Theorem}

\newtheorem{lemma}[theorem]{Lemma}

\newtheorem{definition}[theorem]{Definition}

\newtheorem{example}[theorem]{Example}

\newcommand{\dR}{{\mathbb{R}}}
\newcommand{\dN}{{\mathbb{N}}}
\newcommand{\ep}{{\varepsilon}}

\newcommand{\calS}{{\mathcal S}}
\newcommand{\calT}{{\mathcal T}}

\newcommand{\conv}{{\mathrm{conv}}}
\newcommand{\Conv}{{\mathrm{conv}}}

\usepackage{tikz}
\usetikzlibrary{arrows}
\usetikzlibrary{automata,topaths,patterns}

\usepackage{multicol}

\newcounter{figurecounter}
\begin{document}

\title{Approachability with Constraints%
\thanks{Munk, Solan, and Weinbaum acknowledge the support of the Israel Science Foundation, Grants \#323/13 and \#217/17. Fournier acknowledges the support through the ANR Labex IAST.}}

\author{Ga\"etan Fournier%
\thanks{Universit\'e d'Aix-Marseille, AMSE-GREQAM, 13205 Marseille, France. e-mail: gaetan.fournier@univ-amu.fr},
Eden Kuperwasser%
\thanks{The School of Mathematical Sciences, Tel Aviv
University, Tel Aviv 6997800, Israel. e-mail: kuperwasser@mail.tau.ac.il},
Orin Munk%
\thanks{The School of Mathematical Sciences, Tel Aviv
University, Tel Aviv 6997800, Israel. e-mail: orin25@gmail.com},
Eilon Solan%
\thanks{The School of Mathematical Sciences, Tel Aviv
University, Tel Aviv 6997800, Israel. e-mail: eilons@post.tau.ac.il}, and
Avishay Weinbaum%
\thanks{The School of Mathematical Sciences, Tel Aviv
University, Tel Aviv 6997800, Israel. e-mail: avishay.weinbaum@gmail.com}}

\maketitle
\abstract{We study approachability theory in the presence of constraints.
Given a repeated game with vector payoffs,
we characterize the pairs of sets $(A,D)$ in the payoff space such that Player~1 can guarantee that the long-run average payoff converges to the set $A$, while
the average payoff always remains in $D$.}

\section{Introduction}\label{se:intro}

Approachability theory, which was first introduced in \citet{blackwell},
is an extension of the theory of zero-sum strategic-form games to the situation where the outcome is multidimensional.
In a two-player repeated game in which the outcome is an $n$-dimensional vector,
a target set $A$ in $\dR^n$ is \emph{approachable} by Player~1 if he has a strategy that ensures that the long-run average payoff converges to the set,
whatever strategy Player~2 uses.
\citet{blackwell} provided a geometric condition that ensures that a set is approachable by Player~1.
\citet{Hou} and \citet{Spinat} completed the characterization of approachable sets,
by showing that a set is approachable only if it contains a set that satisfies Blackwell's geometric condition.

Approachability theory was used to study no regret with partial monitoring (see \citet{Perchet2009}, \citet{lehrer2007learning}, and \citet{perchet2014unified}).
In fact, approachability is equivalent to regret minimization and calibration (see \citet{cesa2006prediction} and \citet{Perchet2014}).
The theory was also used to study continuous-time  network flows with capacity and unknown demand (\citet{bauso2010optimization}),
production-inventory problems in both discrete and continuous-time (\citet{khmelnitsky2004parallelism}),
and to construct normal numbers (\citet{lehrernormal}).
The geometric principle that lies behind approachability theory has been studied by \citet{Lehrer2002},
the rate of convergence to the target set was studied in \citet{MannorPerchet2013},
and variants of the basic notion of approachability have been studied by, e.g.,
\citet{Vieille1992},
\citet{LehrerSolan2009},
\citet{ShaniSolan},
\citet{MannorPerchetStoltz2014},
and
\citet{Bausoetal}.

In various situations, in addition to having a target set,
the player has constraints that have to be satisfied.
For example, an investment firm makes daily investment decisions and may have various goals, like maximizing the value of its portfolio,
keeping the value of its portfolio higher than the value of other investment firms,
and attracting investors.
The firm may also have various constraints, like keeping its Sharpe ratio above a certain level,
or never going bankrupt.
A second example concerns fellowships obtained by students in various universities.
In every quarter or semester the student has to dedicate her time to several activities, like studying and spending time with friends and family.
However, to keep receiving the fellowship, the student must keep her average grade above a certain level.

Our goal in this paper is to study approachability in the presence of constraints.
Specifically, we consider a two-player repeated game with vector payoffs, and are given two sets $A$ and $D$ in the payoff space.
The set $A$ is the target set that Player~1 would like to approach,
and the set $D$ represents the set of allowable average payoffs:
after any finite history that occurs with positive probability, the average payoff must be in $D$.
We call this problem \emph{approaching $A$ while remaining in $D$}.

If the set $A$ is approachable by Player~1 while remaining in the set $D$,
then necessarily the set $\overline{A} \cap \overline{D}$ is approachable by Player~1, where $\overline{A}$ and $\overline{D}$ are the closures of $A$ and $D$.
Since the outcome after the first stage must be in $D$,
a second necessary condition is that Player~1 has a safe action; that is, an action $s$ such that when Player~1 plays $s$ the outcome is in $D$,
whatever Player~2 plays.
Our first result is that when the set $D$ is open and convex, these two conditions are also sufficient to ensure that the set $A$ is approachable while remaining in $D$.
We moreover show that the rate of convergence of the average payoff to the set $A$ in our setup is the same as the rate of convergence given by \citet{blackwell} or \citet{MannorPerchet2013},
so that the presence of constraints does not slow down the rate of convergence to $A$.

We then study the case that the set $D$ is not convex, provide two sufficient conditions that guarantee that the set $A$ is approachable while remaining in $D$,
and provide an example that shows that the conditions are not necessary.
This example shows the difficulty in providing a general characterization of the pairs of sets $(A,D)$ such that $A$ is approachable by Player~1 while remaining in $D$.

In Section~\ref{se:model_main} we define the model, define the concept of approachability while remaining in a set, and state and prove our main result.
Section~\ref{se:not_convex} is devoted to the case in which the set $D$ is not convex.
Finally, in Section~\ref{section:open} we mention two open problems.

\section{The Model and the Main Result}\label{se:model_main}

A two-player {\em repeated game with vector-payoffs} is a triplet $(I,J,U)$,
where $I$ and $J$ are finite sets of actions for the two players,
and $U = I \times J \to \dR^n$ is a
vector-payoff matrix.
We assume w.l.o.g.~that payoffs are nonnegative and bounded by 1, that is $0 \leq U_k(i,j) \leq 1$ for every $i \in I$, every $j \in J$, and every $1 \leq k \leq n$.
To eliminate trivial cases we assume that both players have at least two actions: $|I| \geq 2$ and $|J| \geq 2$.

At every stage $t \geq 1$ the two players choose, independently and simultaneously, a pair of actions $(i_t,j_t) \in I \times J$,
each one in his action set, which is observed by both players.%
\footnote{As in \citet{blackwell}, for our results it is sufficient to assume that the players observe the outcome $U(i_t,j_t)$
at every stage $t$.}
A \emph{finite history} of length $t$ is a sequence $h_t  = (i_1,j_1,i_2,j_2,\cdots,i_t,j_t) \in (I \times J)^t$ for some $t \geq 0$.
The empty history is denoted $\emptyset$.
We denote by $H = \cup_{t=0}^\infty (I \times J)^t$ the set of all finite histories.
When $h_t  = (i_1,j_1,i_2,j_2,\cdots,i_t,j_t) \in H$ and $t' \leq t$ we denote by $h_{t'} = (i_1,j_1,i_2,j_2,\cdots,i_{t'},j_{t'})$
the prefix of $h_t$ of length $t'$.

We assume perfect recall, and consequently by Kuhn's Theorem we can restrict attention to behavior strategies.
A (behavior) \emph{strategy} for Player~1
(resp.~Player~2) is a function $\sigma : H \to \Delta(I)$
(resp.~$\tau : H \to \Delta(J)$), where $\Delta(A)$ is the space of
probability distributions over $A=I,J$. We denote by $\mathcal{S}$
and $\mathcal{T}$ the strategy spaces of the players 1 and 2,
respectively.

The set $F$ of \emph{feasible payoffs} is the convex hull of possible one stage payoffs, that is,
\[ F:= conv\{ U(i,j) \colon (i,j) \in I \times J \} \subset \mathbb{R}^n. \]

The set of infinite plays is $H^\infty = (I \times J)^\infty$.
This set,
when supplemented with the sigma-algebra generated by all finite cylinders,
is a measure space.
Every pair of strategies $(\sigma,\tau) \in \mathcal{S} \times \mathcal{T}$
defines a probability distribution $P_{\sigma,\tau}$ over $H$. We denote $E_{\sigma,\tau}$ the expectation with respect to
this probability distribution.

The average payoff vector up to stage $t$ is
\[ g_t := \frac{1}{t}\sum_{k=1}^t U(i_k,j_k). \]
Note that for every $t \in \dN$, the average payoff vector $g_t$ is a random variable
with values in $\dR^n$, whose distribution is determined by the
strategies of both players.
When we wish to calculate the average payoff along a finite history $h_t$ we use the notation $g_t(h_t)$.

Let $d(x,y) := ||x - y||_2$ denote the  euclidean distance between the points $x$ and $y$ in $\dR^n$.

\cite{blackwell} defined the concept of approachable sets in repeated games with vector payoffs.
A subset $A \in \mathbb{R}^n$ is \emph{approachable} by Player~1 if there exists a strategy $\sigma \in \calS$
such that for every $\epsilon>0$ there exists an integer $T$ such that for every strategy $\tau \in \calT$ of Player~2 we have:

$$P_{\sigma,\tau} \Bigl[d(g_t,A)<\epsilon, \forall t\geq T \Bigl] > 1-\epsilon.$$

This paper concerns the concept of approachability with
constraints, which is defined as follows.

\begin{definition}\label{def:approch_constr}
Let $A$ and $D$ be two subsets of $\mathbb{R}^n$.
The set $A$ is \emph{approachable by Player~1 while remaining in the set} $D$
if there exists a strategy $\sigma \in \calS$ such that for every $\epsilon>0$
there exists an integer $T \in \dN$ such that for every strategy $ \tau \in \calT$ of Player~2, we have:

\begin{eqnarray}
\label{equ:1}
P_{\sigma,\tau} \Bigl[\forall t\geq T, ~d(g_t,A)<\epsilon \Bigl] &>& 1-\epsilon,\\
P_{\sigma,\tau} \Bigl[ \forall t \geq 1,~ g_t \in D \Bigl]&=&1.
\label{equ:2}
\end{eqnarray}
\end{definition}

Condition (\ref{equ:1}) ensures that the strategy $\sigma$ approaches the set $A$.
Condition (\ref{equ:2}) is concerned with the constraints:
when playing $\sigma$, Player~1 guarantees that the sequence of realized average payoffs always remains in the set $D$.
Our main goal is the characterization of the pairs of sets $(A,D)$ such that $A$ is approachable by Player~1 while remaining in $D$.

For every mixed action $p \in \Delta(I)$ define
\[ R_1(p):=\{ U(p,q) \colon q \in \Delta(J)\} = \conv\{ U(p,j) \colon j \in J\}. \]
This is the set of all possible expected outcomes when Player~1 plays the mixed action $p$.
For every set $X$ in a Euclidean space we denote by $\overline{X}$ the closure of $X$.

The following lemma lists two necessary conditions to approaching $A$ while remaining in $D$.
\begin{lemma}
Let $A$ and $D$ be two subsets of $\dR^n$.
If the set $A$ is approachable by Player~1 while remaining in the set $D$, then the following two conditions hold.
\begin{enumerate}
\item[(C1)]
The set $\overline{A} \cap \overline{D}$ is approachable by Player~1.
\item[(C2)]
There exists an action $s \in I$ such that for every action $j \in J$ we have $U(s,j) \in D$.
\end{enumerate}
\end{lemma}
An action $s \in I$ that satisfies condition (C2) is termed a \emph{safe action},
since it ensures that the stage payoff is in $D$.

\bigskip

\begin{proof}
We first argue that Condition (C1) is necessary.
Suppose that the set $A$ is approachable by Player~1 while remaining in the set $D$,
and let $\sigma$ be a strategy for Player~1 that guarantees that the average payoff converges to $A$
while remaining in $D$.
By Eq.~(\ref{equ:1}), any accumulation point of the sequence of average payoffs lies in $\overline A$, regardless of the strategy used by Player~2.
Since the average payoff remains in $D$, any accumulation point of the sequence of average payoffs lies in $\overline D$.
Hence any accumulation point of the sequence of average payoffs lies in the set $\overline A \cap \overline D$,
which implies that the set $\overline A \cap \overline D$ is approachable by Player~1.

We now argue that Condition (C2) is necessary.
Denote by $s$ any action that the strategy $\sigma$ plays with positive probability at the first stage.
Let $\tau$ be any strategy for Player~2 that plays all actions in $J$ with positive probability.
Under the strategy pair $(\sigma,\tau)$ the probability that the action pair $(s,j)$ is played is positive,
and therefore
Eq.~(\ref{equ:2}) implies that $U(s,j) \in D$.
\end{proof}

Our main result is a characterization of the pairs of sets $(A,D)$ such that the set $A$ is approachable while remaining in $D$,
and it is valid whenever the set $D$ is open and convex.
Convexity is a natural assumption,
as often constraints have the form of linear inequalities.
The requirement that the set $D$ is open means that these inequalities should be strict.
The characterization states that when the set $D$ is open and convex,
Conditions~(C1) and~(C2) are also sufficient.

\begin{theorem}\label{thm:main_thm}
Let $A$ be a subset of $\mathbb{R}^n$ and let $D$ be an open and convex subset of $\mathbb{R}^n$.
Player~1 can approach $A$ while remaining in $D$ if and only if Conditions~(C1) and~(C2) hold.
\end{theorem}

We note that Condition (C2), together with the assumption that $D$ is convex,
implies that $R_1(s) \subseteq D$.
Because both $R_1(s)$ and $F \setminus D$ are closed, we have $\delta:=d(F \setminus D,R_1(s))>0$.

The following example shows that when $D$ is not open,
Conditions~(C1) and~(C2) are not sufficient to imply that $A$ is approachable while remaining in $D$.

\begin{figure}
    \label{figure:example1}
    \center{
    $\begin{array}{c|c|c|}
    & L & R \\
    \hline
    T & (0,1) & (0,1) \\
    \hline
    B_1 &  (1,0) & (-1,0)\\
    \hline
    B_2 & (-1,0)  & (1,0)\\
    \hline
    \end{array}$ \ \ \ \ \ \ \ \ \ \
    \includegraphics{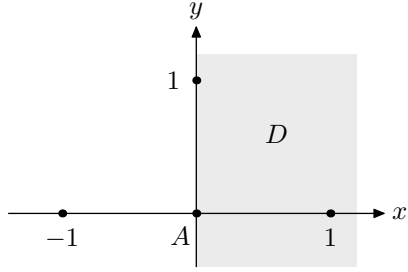}}
    \caption{The payoff matrix in Example \ref{example:11} and the payoff space.}
\end{figure}

\begin{example}
\label{example:11}
Consider the game that is depicted in Figure~\arabic{figurecounter} with $A = (0,0)$ and $D= \{ (x,y) \in \dR^2 \colon x \geq 0\}$.
For convenience we consider, in this example and in the following ones, payoffs that do not necessarily belong to the interval $[0,1]$.

The set $A = \overline{A} \cap \overline{D}$ is approachable by Player~1, for example,
by the stationary strategy $[\tfrac{1}{2}(B_1),\tfrac{1}{2}(B_2)]$,
and therefore Condition (C1) holds.
The action $T$ is a safe action for Player~1, and therefore Condition (C2) holds.

We argue that the set $A$ is not approachable by Player~1 while remaining in $D$.
Indeed, to approach $A$ Player~1 has to play one of the actions $B_1$ and $B_2$ at least once
(in fact, with probability 1 he should play these action infinitely often).
Suppose that Player~2 plays the stationary strategy $[\tfrac{1}{2}(L),\tfrac{1}{2}(R)]$.
Then in the first stage in which Player~1 plays $B_1$ or $B_2$, there is a probability $\tfrac{1}{2}$ that the outcome is $(-1,0)$,
in which case the first coordinate of the average payoff is negative,
so that the average payoff is not in $D$.
\end{example}
\addtocounter{figurecounter}{1}

Example \ref{example1} below shows that when $D$ is not convex,
Conditions (C1) and (C2) are not sufficient to imply that $A$ is approachable while remaining in $D$.

The rest of this section is devoted to the proof of Theorem~\ref{thm:main_thm}.

\subsection{$B$-Sets and Blackwell's Characterization of Approachable Sets}

\citet{blackwell} provided a geometric characterization for approachable sets (without constraints).
The basic concept that Blackwell used was that of $B$-sets.
In this subsection we review the definition of $B$-sets and Blackwell's characterization.

A \emph{hyperplane} in $\dR^n$ is any set of the form $\mathcal{H} := \{x \in \dR^n \colon \sum_{k=1}^n \alpha^k x^k = \beta\}$,
where $\alpha^1,\cdots,\alpha^k,\beta \in \dR$.
For every hyperplane $\mathcal{H}$ 
 we denote by
$\mathcal{H}^+ := \{x \in \dR^n \colon \sum_{k=1}^n \alpha^k x^k \geq \beta\}$ and
$\mathcal{H}^- := \{x \in \dR^n \colon \sum_{k=1}^n \alpha^k x^k \leq \beta\}$.
These are the two half-spaces defined by $\mathcal{H}$.

\begin{definition}\label{def:F_Bset}
A set $A \subset \mathbb{R}^n$ is a \emph{$B$-set} for Player~1 if for every point $x \in F \setminus A$
there exists a point $y \in \overline{A}$ and a mixed action $p \in \Delta(I)$ such that
(a) $y$ minimizes the distance to $x$ among the points in $A$, and
(b) the hyperplane $\mathcal{H}$ that is (i) perpendicular to the line that connects $x$ to $y$ and (ii) passes through $y$, separates $x$ from $R_1(p)$,
that is, $x \in \mathcal{H}^-$ and $R_1(p) \subseteq \mathcal{H}^+$, or $x \in \mathcal{H}^+$ and $R_1(p) \subseteq \mathcal{H}^-$.
\end{definition}

\cite{blackwell} proved that every B-set is approachable.
\cite{Hou} and \cite{Spinat} proved that every approachable set contains a B-set.

In addition, for every B-set $A$ for Player~1 Blackwell identified a strategy for Player~1 that guarantees that the average payoff converges to $A$ at a rate of
$O(\frac{1}{\sqrt{t}})$, where $t$ is the number of stages played so far; that is, there is a strategy $\sigma$ for Player~1
such that for every $\ep > 0$ there is a constant $c > 0$ such that for every strategy $\tau$ of Player~2 and every $t \in \dN$,
\begin{equation}
\label{equ:rate}
P_{\sigma,\tau}[d(g_t,A) < \tfrac{c}{\sqrt{t}}] \geq 1-\ep,
\end{equation}
see, e.g., \cite{cesa2006prediction}, Remark 7.7 and Exercise 7.23.

\subsection{A definition of a strategy $\sigma^*$.}
We now define a strategy $\sigma^*$ for Player~1 that, as will be shown later, approaches $A$ while remaining in $D$.
This strategy is based on two components:
a strategy $\widehat\sigma$ that approaches the set $\overline{A} \cap \overline{D}$,
and a set $H^*$ of finite histories that we now define.
Roughly, the strategy $\sigma^*$ follows the strategy $\widehat\sigma$ that approaches the set $\overline{A} \cap \overline{D}$,
but whenever the average payoff gets close to the boundary of $D$,
it plays the safe action, ensuring that the average payoff gets farther away from the boundary.

Let $H^* \subseteq H$ be the set of all finite histories $h_t \in H$ that satisfy $d(g_t(h_t),F \setminus D) \leq \tfrac{3}{t}$.
We also add the empty history $\emptyset$ to $H^*$.
Thus, the complement $H \setminus H^*$ contains all finite histories $h_t \in H$ that satisfy $d(g_t(h_t),F \setminus D) > \tfrac{3}{t}$.
This implies that whatever Player~1 plays after a history $h_t \not\in H^*$, the average payoff after stage $t+1$ will be in $D$.

For every finite history $h_t \in H$ denote by $\varphi(h_t)$ the part of $h_t$ that is played after stages $t'$ such that $h_{t'} \notin H^*$.
 Formally, we define $\varphi$ recursively as follows.
\begin{eqnarray}
\varphi(\emptyset) &:=& \emptyset,\\
\varphi(h_t,i_{t+1},j_{t+1}) &:=& \left\{
\begin{array}{lll}
\varphi(h_t) & \ \ \ \ \ & \hbox{if } h_t \in H^*,\\
\varphi(h_t) \circ (i_{t+1},j_{t+1}) & & \hbox{if } h_t \not\in H^*,
\end{array}
\right.
\end{eqnarray}
where $\varphi(h_t) \circ (i_{t+1},j_{t+1})$ is the concatenation of the finite history $\varphi(h_t)$ with the action pair $(i_{t+1},j_{t+1})$.

We turn to the definition of the strategy $\sigma^*$.
Let $\sigma^*$ be the strategy of Player~1 that plays the safe action $s$ whenever $h_t \in H^*$,
and plays the strategy that approaches $\overline{A} \cap \overline{D}$ whenever $h_t \not\in H^*$,
ignoring the stages that were played after subhistories in $H^*$. Formally,
\[ \sigma^*(h_t) :=
\left\{
\begin{array}{lll}
s & \ \ \ \ \ & \hbox{if } h_t \in H^*,\\
\widehat\sigma(\varphi(h_t)) & & \hbox{if } h_t \not\in H^*.
\end{array}
\right.
\]
We will prove that the strategy $\sigma^*$ approaches $A$ while remaining in $D$.
This will be done in four steps.
\begin{enumerate}
\item   In Lemma~\ref{prop_inD} we prove that when Player~1 plays $\sigma^*$ the average payoff always remains in $D$,
whatever Player~2 plays.
\item   In Lemma~\ref{lemma:uniform:1} we prove that the frequency of stages in which $\sigma^*$ plays the safe action goes to 0.
\item   In Lemma~\ref{lemma:distance_bound} we prove a geometric inequality used in the proof of Lemma \ref{lemma:uniform:1}.
\item   In Lemma~\ref{lemma:uniform:2} we prove that the strategy $\sigma^*$ approaches the set $A$.
\end{enumerate}

\begin{lemma}\label{prop_inD}
For any strategy $\tau$ of player 2, we have:
$$P_{\sigma^*,\tau}\left[\exists t \geq 1, ~ g_t \in D\right]=0.$$
\end{lemma}

The properties that are needed to prove Lemma~\ref{prop_inD} are that
(a) after histories in $H^*$ the strategy $\sigma^*$ plays a safe action $s$
and (b) the set $D$ is convex.


\begin{proof}
Fix a strategy $\tau$ of Player~2.
For every finite history $h_t \in H$
denote by $P_{\sigma,\tau}(h_t)$ the probability that under the strategy pair $(\sigma,\tau)$ the realized history of length $t$ is $h_t$.
We will prove by induction on $t$ that for every history $h_t \in H$ that satisfies $P_{\sigma,\tau}(h_t) > 0$ we have $g_t(h_t) \in D$.

Since at the first stage the strategy $\sigma$ plays a safe action, the claim holds for $t=1$.
Assume then by induction that the claim holds for $t-1$, and let $h_t$ be a history of length $t$ with $P_{\sigma,\tau}(h_t) > 0$.
In particular, $h_{t-1}$, the prefix of $h_t$ of length $t-1$, satisfies $P_{\sigma,\tau}(h_{t-1}) > 0$,
and consequently by the induction hypothesis we have $g_{t-1}(h_{t-1}) \in D$.

If $h_{t-1} \in H^*$ then at stage $t$ the strategy $\sigma$ plays a safe action.
Since the set $D$ is convex, since $g_{t-1}(h_{t-1}) \in D$, and since $R_1(s) \subseteq D$,
it follows that $g_t(h_t) \in D$.

If $h_{t-1} \not\in H^*$ then $d(g_{t-1}(h_{t-1}),F \setminus D) > \tfrac{3}{t-1}$.
Because payoffs are bounded by $1$, we have that $d(g_{t-1}(h_{t-1}),g_t(h_t))\leq \frac{2}{t}$.
This implies that $d(g_t(h_t),F \setminus D) > 0$, so that $g_t(h_t) \in D$ as well.
\end{proof}

Denote by $f(h_t)$ the number of times along the finite history $h_t$ in which the history $h_{t'}$ belongs to $H^*$ for $t' < t$, that is,
\[ f(h_t) := \#\{ 0 \leq t' < t \colon h_{t'} \in H^*\}. \]
For every finite history $h_t \in H$, we denote the average payoff during the stages in which the history was in $H^*$ by
\[ \alpha_t := \frac{1}{f(h_t)} \sum_{0 \leq t' < t, h_{t'} \in H^*} U(i_{t' + 1},j_{t' + 1}) =  \frac{1}{f(h_t)} \sum_{0 \leq t' < t, h_{t'} \in H^*} U(s,j_{t' + 1}), \]
and the average payoff up to stage $t$ during the stages in which the history was not in $H^*$ by
\[ \beta_t := \frac{1}{t-f(h_t)} \sum_{0 \leq t' < t, h_{t'} \not\in H^*} U(i_{t' + 1},j_{t' + 1}). \]
Note that $\alpha_t \in R_1(s)$, so in particular $d(\alpha_t,F \setminus D) \geq \delta$.
Note also that
\begin{equation}
\label{equ:payoff}
g_t = \frac{1}{t}\sum_{k=1}^t U(i_k,j_k) = \frac{f(h_t)}{t} \alpha_t + \frac{t-f(h_t)}{t} \beta_t.
\end{equation}


The following result, which provides a uniform upper bound on $f(h_t)$,
implies that along the play,
the safe action is played relatively rarely.

\begin{lemma}
\label{lemma:uniform:1}
For every $\ep > 0$ there exists a constant $c' \in \dR^{+}$ such that for every strategy $\tau$ of Player~2 we have
\[ P_{\sigma^*,\tau}\left[f(h_t) \leq \tfrac{c'\sqrt{t}}{\delta}, \ \ \ \forall t \in \mathbb{N}\right] \geq  1-\ep. \]
\end{lemma}

To prove Lemma~\ref{lemma:uniform:1} we need the following technical result.

\begin{lemma}
    \label{lemma:distance_bound}
    For every $x \in D$, every $y \in \dR^n$, and every $\lambda \in [0,1]$, we have
    $$d( \lambda x + (1-\lambda) y , F \setminus D ) \geq \lambda d(x, F \setminus D) - (1-\lambda) d(y, D)$$
\end{lemma}

\begin{proof}

\noindent\textbf{Step 1:} Definitions.

Define two continuous functions $g_1,g_2 : [0,1] \to \dR$ by
$$ g_1(\lambda) := d( \lambda x + (1-\lambda) y , F \setminus D ), \ \ \ \forall \lambda \in [0,1], $$
and
$$ g_2(\lambda) := \lambda d(x, F \setminus D) - (1-\lambda) d(y, D).$$
Since the set $D$ is open, $g_1(1) = g_2(1) = d(x,F \setminus D) > 0$.
Define $$\lambda_0 := \inf\{\lambda \in [0,1] \colon g_1(\lambda) > 0\}.$$

\smallskip
\noindent\textbf{Step 2:} The function $g_1$ is concave on $(\lambda_0,1]$.

The claim holds since the set $D$ is convex.
Indeed, suppose that $\lambda,\lambda' \in (\lambda_0,1]$, $g_1(\lambda) = c>0$, and $g_1(\lambda') =c'>0$.
This implies that $d(\lambda x + (1-\lambda)y,F \setminus D) = c$ and $d(\lambda' x + (1-\lambda')y,F \setminus D) = c'$.
Consequently, $B(\lambda x+(1-\lambda)y,c) \subseteq D$ and
$B(\lambda'x+(1-\lambda')y,c') \subseteq D$,
where $B(z,r)$ is the open ball around $z$ with radius $r$, for every $z \in \dR^n$ and every $r \geq 0$.
It follows that $B(\lambda''x+(1-\lambda'')y,c'') \subseteq D$,
where $\lambda'' := \tfrac{1}{2}\lambda + \tfrac{1}{2}\lambda'$   and $c'' := \tfrac{1}{2}c+\tfrac{1}{2}c'$.
Therefore $g_1(\lambda'') = d(\lambda''x+(1-\lambda'')y,F \setminus D) \geq c''$.
The function $g_1$ is therefore mid-point concave and continuous, hence concave.

\smallskip
\noindent\textbf{Step 3:} $g_2(\lambda) \leq 0$ for every $\lambda \in [0,\lambda_0)$.

The claim holds trivially whenever $\lambda_0=0$.
We therefore assume that $\lambda_0 > 0$, so in particular $y$ is not in $D$.
We will show that $g_2(\lambda_0) \leq 0$.
The result for every $\lambda \in [0,\lambda_0)$ will follow since the function $g_2$ is monotone increasing.

Set $q := \lambda_0 x + (1-\lambda_0) y$. Then $q$ lies on the boundary of the set $D$.
Denote by $z$ the unique closest point to $y$ in $D$ (see Figure~\arabic{figurecounter}; the uniqueness follows since $D$ is convex).
Denote by $\theta$ the angle between the line segment $[y,z]$ and the line segment $[y,x]$.
If $\theta > 0$, denote by $w$ the intersection point of the half line $[z,q]$, and the half line that starts at $x$,
lies on the plane defined by $x$, $y$, and $z$,
and has angle $\theta$ relative to the line segment $[x,y]$. The triangles $(x,w,q)$ and $(y,z,q)$ are similar,
hence $(1-\lambda_0)d(y,z) = \lambda_0d(x,w)$.
If $\theta = 0$, define $w := z$, and then $(1-\lambda_0)d(y,z) = \lambda_0d(x,w)$ holds as well.
We conclude that
$$(1-\lambda_0)d(y,D) = (1-\lambda_0)d(y,z) = \lambda_0d(x,w) \geq \lambda_0d(x,F \setminus D),$$
where the last inequality follows from the fact that $w$ is in $F \setminus D$,
because $D$ is convex. Consequently, $g_2(\lambda_0) \leq 0$ as desired.

\smallskip
\noindent\textbf{Step 4:} $g_1(\lambda) \geq g_2(\lambda)$ for every $\lambda \in [0,1]$.

For $\lambda \in [0,\lambda_0)$ we have by Step 3
\[ g_1(\lambda) = 0 \geq g_2(\lambda). \]
By the continuity of $g_1$ and $g_2$ this inequality extends to $\lambda=\lambda_0$.
By Step 2 the function $g_1$ is concave on $(\lambda_0,1]$ and by its definition the function $g_2$ is linear on this interval.
Since $g_1(1) = g_2(1)$ while  $g_1(\lambda_0) \geq g_2(\lambda_0)$,
it follows that $g_1(\lambda) \geq g_2(\lambda)$ for every $\lambda \in [\lambda_0,1]$.

\begin{figure}
\center{
    \includegraphics{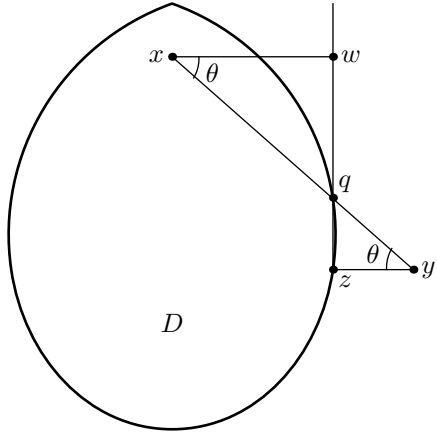}
   \caption{The construction in the proof of Lemma \ref{lemma:distance_bound}.}}
\end{figure}
\end{proof}
\addtocounter{figurecounter}{1}

\begin{proof}[Proof of Lemma~\ref{lemma:uniform:1}]
It is sufficient to prove the claim for histories $h_t \in H$ that satisfy $f(h_t) \geq 2$.
Fix then $\ep > 0$, a strategy $\tau$ of Player~2, and a finite history $h_t \in H$ that satisfies $f(h_t) \geq 2$.
Denote by $m = m(h_t)$ the last stage along $h_t$ such that $h_m \in H^*$.
Note that if $m < t$ then $f(h_t) = f(h_m) + 1$, while if $m=t$ then $f(h_t)=f(h_m)$.
Since $h_m \in H^*$ we have $d(g_m,F \setminus D) \leq \tfrac{3}{m}$.
By Eq.~(\ref{equ:payoff}) and Lemma~\ref{lemma:distance_bound} we deduce that
\begin{eqnarray}
\nonumber
\tfrac{3}{m} &\geq& d(g_m,F \setminus D)\\
\nonumber
 &=&  d\left(\tfrac{f(h_m)}{m} \alpha_m + \tfrac{m-f(h_m)}{m}\beta_m,F \setminus D\right)\\
 &\geq& \tfrac{f(h_m)}{m}d(\alpha_m, F \setminus D) - \tfrac{m-f(h_m)}{m}d(\beta_m, D).
 \label{equ:payoff:2}
\end{eqnarray}
Since $\alpha_m \in R_1(s)$ we have $d(\alpha_m, F \setminus D) \geq \delta$.
Since $\beta_m$ is the average payoff when playing the approachability strategy for $m-f(h_m)$ stages,
there is a constant $c \geq 3$, such that%
\footnote{To properly interpret Eq.~(\ref{equ:54}) and the event $\left\{\beta_m \leq \tfrac{c}{\sqrt{m - f(h_m)}}, \ \ \ \forall t \in \mathbb{N}\right\}$, recall that $m$ depends on $t$.}
\begin{equation}
\label{equ:54}
P_{\sigma^*,\tau}\left[\beta_m \leq \tfrac{c}{\sqrt{m - f(h_m)}}, \ \ \ \forall t \in \mathbb{N}\right] \geq 1-\ep.
\end{equation}
Together with Eq.~(\ref{equ:payoff:2}) this implies that on the event $\left\{\beta_m \leq \tfrac{c}{\sqrt{m - f(h_m)}}, \ \ \ \forall t \in \mathbb{N}\right\}$ we have
\begin{eqnarray}
    \nonumber
    \tfrac{c}{m} \geq \tfrac{3}{m} &\geq& \tfrac{f(h_m)}{m}\delta - \tfrac{m-f(h_m)}{m} \cdot \tfrac{c}{\sqrt{m - f(h_m)}}\\
    &=& \tfrac{f(h_m)}{m}\delta - \tfrac{c\sqrt{m-f(h_m)}}{m}\\
    &\geq& \tfrac{f(h_m)}{m}\delta - \tfrac{c\sqrt{m}}{m},
\end{eqnarray}
which solves to $f(h_m) \leq \tfrac{c(1+\sqrt{m})}{\delta}$.
Consequently, on this event
\[ f(h_t) \leq f(h_m) + 1 \leq \tfrac{c(1+\sqrt{m})}{\delta} + 1 \leq \tfrac{c(2+\sqrt{t})}{\delta}, \]
and the result follows.
\end{proof}

We now complete the proof of Theorem~\ref{thm:main_thm} by showing that the strategy $\sigma^*$ approaches the set $\overline{A} \cap \overline{D}$,
and, in particular, approaches the set $A$.
This result holds since in most stages Player~1 plays a strategy that approaches the set $\overline{A} \cap \overline{D}$.

\begin{lemma}
\label{lemma:uniform:2}
The strategy $\sigma^*$ approaches the set $\overline{A} \cap \overline{D}$.
\end{lemma}

\begin{proof}
Fix $\ep > 0$.
Since the strategy $\widehat\sigma$ approaches the set $\overline{A} \cap \overline{D}$,
there is $T_0 \in \mathbb{N}$ such that for every strategy $\tau$ of Player~2,
\begin{equation}
\label{equ:21}
P_{\widehat\sigma,\tau} \Bigl[d(g_t,\overline{A} \cap \overline{D})<\epsilon, \forall t\geq T_0 \Bigl] > 1-\epsilon.
\end{equation}

Consider an outside observer who wakes up at every stage $t$ such that $h_{t-1} \not\in H^*$.
This observer is not aware of the stages $t$ such that $h_{t-1} \in H^*$ and from his point of view Player~1 follows the strategy $\widehat\sigma$.
Let $c'$ be the constant of Lemma~\ref{lemma:uniform:1},
and let $T_1 \in \dN$ be sufficiently large so that $T_1 - \tfrac{c'\sqrt{T_1}}{\delta} \geq T_0$.
By Lemma~\ref{lemma:uniform:1}, with high probability the number of stages which the observer missed up to stage $T_1$ is at most $\tfrac{c'\sqrt{T1}}{\delta}$.
Hence, up to stage $T_1$ of the actual game, with high probability the observer observed at least $T_0$ stages.

Recall that $\beta_t$ is the average payoff up to stage $t$ during the stages $m$ in which the partial history up to stage $m$ is not in $H^*$;
that is,
this is the average payoff as observed by the observer.
Let $\Omega'$ be the event that, for the observer,
$d(\beta_t,\overline{A} \cap \overline{D}) \leq \ep$ for every $t\geq T_0$:
\[ \Omega' := \{ d(\beta_t,\overline{A} \cap \overline{D}) \leq \ep, \forall t \hbox{ such that } h_t \not\in H^*, t-f(h_t) \geq T_0\}. \]
By Eq.~(\ref{equ:21}) we have $P_{\sigma^*,\tau}(\Omega') > 1-\ep$.

Denote by $\Omega''$ the event
\[ \Omega'' := \Omega' \cap \left\{f(h_t) \leq \tfrac{c'\sqrt{t}}{\delta}, \forall t \in \dN\right\}. \]
By Lemma~\ref{lemma:uniform:1} we have $P_{\sigma^*,\tau}[\Omega''] \geq 1-2\ep$.

From now on we restrict our attention to the event $\Omega''$ and we fix $t \geq T_1$.
By definition we have $d(\beta_t, \overline{A} \cap \overline{D}) \leq \ep$.
Since payoffs are in between 0 and $1$, we have $d(\alpha_t,\overline{A} \cap \overline{D}) \leq 1$, which implies that
\begin{equation}
\label{equ:uniform2}
d(g_t,\overline{A} \cap \overline{D}) \leq
\tfrac{f(h_t)}{t} + \tfrac{t - f(h_t)}{t}d(\beta_t,\overline{A} \cap \overline{D}) \leq \ep + \tfrac{c'}{\delta \sqrt{t}}.
\end{equation}
Taking $T_2 := \max{\{T_1, \tfrac{c'^2}{\ep^2 \delta^2}\}}$ we obtain that on $\Omega''$
\[ d(g_t,\overline{A} \cap \overline{D}) \leq 2\ep, \ \ \ \forall t \geq T_2, \]
and the desired result follows.
\end{proof}

We now show that the rate of convergence to the target set $A$ is not harmed by the introduction of constraints.
In particular, remaining in $D$ as part of the approachability strategy does not incur additional penalties on the rate of approachability
from an asymptotic perspective.

\begin{lemma}
\label{corollary:rate}
The rate at which the strategy $\sigma^*$ approaches $\overline{A} \cap \overline{D}$ is $O(\frac{1}{\delta\sqrt{t}})$;
that is, for every $\epsilon>0$ there exists a constant $c'' > 0$ such that for every strategy $\tau$ of Player~2
and every $t \in \dN$ we have
$$ P_{\sigma^{\star},\tau} \left[ d(g_t,\overline{A} \cap \overline{D}) < \frac{c''}{\sqrt{t}} \right] \geq 1- \epsilon.$$
   \end{lemma}

\begin{proof}
We use the notations of Lemma~\ref{lemma:uniform:2}.
Fix $\ep > 0$.
Recall that $\beta_t$ is the average payoff up to stage $t$ in those stages in which the strategy $\sigma^*$ followed the strategy $\widehat \sigma$ that
approaches $\overline A \cap \overline D$.
Hence, there is $c > 0$ such that with probability at least $1-\ep$ we have $d(\beta_t,\overline{A} \cap \overline{D}) \leq \tfrac{c}{\sqrt{t - f(h_t)}}$ for every $t \in \dN$
and every strategy $\tau$ of Player~2 (see Eq.~(\ref{equ:rate})).
By Eq.~(\ref{equ:uniform2}) and Lemma~\ref{lemma:uniform:1}, with probability at least $1-\ep$ we have
\begin{eqnarray}
\nonumber
d(g_t,\overline{A} \cap \overline{D}) &\leq& \tfrac{t - f(h_t)}{t} d(\beta_t,\overline{A} \cap \overline{D}) + \tfrac{f(h_t)}{t}\\
&\leq& \tfrac{c \sqrt{t - f(h_t)}}{t} + \tfrac{c' \sqrt{t}}{\delta t} \\
&\leq& \tfrac{c}{\sqrt{t}} + \tfrac{c'}{\delta \sqrt{t}} = \left(c + \frac{c'}{\delta}\right)\frac{1}{\sqrt{t}},
\end{eqnarray}
and the claim follows.
\end{proof}

\section{The Case that $D$ is not Convex}\label{se:not_convex}

The proof of Theorem \ref{thm:main_thm} hinges on the assumption that the set $D$ is convex.
In this section we study approachability with constraints when the set $D$ is not convex.
We start with an example, which shows that in the absence of convexity, Conditions~(C1) and~(C2)
are not sufficient to guarantee that Player~$1$ can approach $A$ while remaining in $D$.
This example will lead us to a weaker concept of approachability with constraints that we will examine.

\begin{example}
\label{example1}
Consider the game that appears in Figure~\arabic{figurecounter},
where Player~1 has four actions, $T_1$, $T_2$, $B_1$, and $B_2$, Player~2 has two actions, $L$, and $R$,
and the payoff is two-dimensional.

\[ \begin{array}{c|c|c|}
 & L & R \\
\hline
T_1 & (1,2) & (2,2) \\
\hline
T_2 &  (2,2) & (1,2)\\
\hline
B_1 &  (1,1) & (2,1)\\
\hline
B_2 & (2,1)  & (1,1)\\
\hline
\end{array}\]

\centerline{Figure \arabic{figurecounter}: The payoff matrix in Example \ref{example1}.}
\addtocounter{figurecounter}{1}

\bigskip

Let $0<\alpha'<\alpha<\frac{1}{2}$ and define $A:= B((\frac{3}{2},1),\alpha')$ and
  $D:= ([1-\alpha,2+\alpha] \times [2-\alpha,2+\alpha] \cup [\frac{3}{2}-\alpha,\frac{3}{2}+\alpha] \times [1-\alpha,2+\alpha])$,
see Figure~\arabic{figurecounter}.

\begin{figure}[h!]
\begin{center}
\begin{tikzpicture}
[scale=1,
    axis/.style={very thick, ->, >=stealth'},
    dashed line/.style={dashed, thin},
    pile/.style={thick, ->, >=stealth', shorten <=2pt, shorten
    >=2pt},
    every node/.style={color=black}
    ]

    \draw[axis] (0.5,0.4)  -- (0.5,5.5) ;
    \draw[axis] (0.4,0.5) -- (5.5,0.5) ;

    \begin{scope}[shift={(6,6)},rotate=180]
   \draw[-] (1.5,1.5)  -- (1.5,2.5);
   \draw[-] (1.5,2.5)  -- (2.5,2.5);
   \draw[-] (2.5,2.5)  -- (2.5,4.5);
   \draw[-] (2.5,4.5)  -- (3.5,4.5);
   \draw[-] (3.5,4.5)  -- (3.5,2.5);
   \draw[-] (3.5,2.5)  -- (4.5,2.5);
   \draw[-] (4.5,2.5)  -- (4.5,1.5);
   \draw[-] (4.5,1.5)  -- (1.5,1.5);
   \draw (3,4) circle[radius=10pt];
   \node[] at (3,4){$A$};

  \end{scope}

   \node[] at (3,3.5){$D$};

\draw[-] (2,0.4)--(2,0.6) ;
\node[] at (2,0.25){1};
\draw[-] (4,0.4)--(4,0.6) ;
\node[] at (4,0.25){2};
\draw[-] (0.4,2)--(0.6,2) ;
\node[] at (0.2,2){1};
\draw[-] (0.4,4)--(0.6,4) ;
\node[] at (0.2,4){2};

\draw (2,2) circle[radius=2pt];
\draw (4,2) circle[radius=2pt];
\draw (2,4) circle[radius=2pt];
\draw (4,4) circle[radius=2pt];
\end{tikzpicture}
\end{center}
\centerline{Figure \arabic{figurecounter}: The sets $A$ and $D$ in Example \ref{example1}.}
\end{figure}

\addtocounter{figurecounter}{1}

\bigskip

Conditions~(C1) and~(C2) are satisfied;
indeed, the actions $T_1$ and $T_2$ are safe actions,
and the strategy $\frac{1}{2}B_1+\frac{1}{2}B_2$ approaches the set $A$, which is a strict subset of $D$,
and therefore it also approaches the set $\overline{A} \cap \overline{D}$.

Player~1 cannot approach $A$ while remaining in $D$.
Indeed, assume to the contrary that Player~1 has a strategy $\sigma$ that approaches $A$ while remaining in $D$,
and let $\tau$ be the stationary strategy of Player~2 that plays $L$ and $R$ with equal probability at every stage.
Fix $\ep > 0$ sufficiently small and suppose that the average payoff at stage $t_0$ is in the set $B(A,\ep)$.
It might happen with positive probability, albeit small, that in the next $\tfrac{t_0}{\ep}$ stages the first coordinate of the outcome is 1.
In this case there will be $t > t_0$ such that the average payoff at stage $t$ is not in $D$.

Nevertheless, as we now argue, Player~1 can approach $A$ while remaining in $D$ with high probability.
To do so, Player~$1$ plays the mixed action $\frac{1}{2}T_1+\frac{1}{2}T_2$ for $K$ stages, where $K \in \dN$ is sufficiently large,
and afterwards he plays the mixed action $\frac{1}{2}B_1+\frac{1}{2}B_2$.
During the first $K$ stages the average payoff is in the convex hull of $(1,2)$ and $(2,2)$, and in particular it remains in $D$.
Moreover, by the strong law of large numbers, provided $K$ is sufficiently large, with high probability the first coordinate of the average payoff at stage $K$ is between
$\tfrac{3}{2}-\tfrac{\alpha}{2}$ and $\tfrac{3}{2}+\tfrac{\alpha}{2}$.
By the strong law of large numbers once again, and provided $K$ is sufficiently large,
with high probability the first coordinate of the average payoff at every stage $t \geq K$ is between $\tfrac{3}{2}-\alpha$ and $\tfrac{3}{2}+\alpha$,
in which case the average payoff remains in $D$.
It follows that Player~1 can indeed approach $A$ while remaining in $D$ with high probability.
\end{example}

Example~\ref{example1} leads us to the study of probabilistic approachability with constraints.

\begin{definition}\label{def:weak_approch_constr}
Let $A$ and $D$ be two subsets of $\mathbb{R}^n$.
Given $\ep > 0$, we say that Player~1 \emph{can approach $A$ while remaining in $D$ with probability at least $1-\ep$}
if there exist a strategy $\sigma_\ep$ and an integer $T_\ep$ such that for every strategy $\tau$ of Player~2 we have:
\begin{eqnarray}
P \Bigl[d(g_t,A)<\epsilon, \forall t\geq T_\ep \Bigl] > 1-\epsilon,\\
P \Bigl[ d(g_t,F \setminus D)>0, \forall t \geq 0 \Bigl]> 1-\epsilon.
\end{eqnarray}
We say that Player~1 \emph{can approach $A$ while remaining in $D$ with high probability}
if for every $\epsilon>0$ Player~1 can approach $A$ while remaining in $D$ with probability at least $1-\ep$.
\end{definition}

To prove Theorem \ref{thm:main_thm}, which studied the case in which the set $D$ is convex,
we constructed a strategy that ``directly" approaches $\overline{A} \cap \overline{D}$:
the strategy attempted to approach the set $\overline{A} \cap \overline{D}$, and played a safe action only to ensure that the average payoff does not leave $D$.
In the next example we illustrate a more complex strategy that handles the nonconvexity of the set $D$
by setting intermediate goals to Player~1.

\begin{example}
\label{example:2}
Consider the game that appears in Figure~\arabic{figurecounter},
where Player~1 has four actions, $x_0$, $x_1$, $x_2$, and $x_3$, Player~2 has two actions, $L$ and $R$,
and the payoff is two-dimensional.

\[ \begin{array}{c|c|c|}
 & L & R \\
 \hline
x_0 & (1,1) & (1,1)\\
 \hline
x_1 & (4,1) & (4,1)\\
 \hline
x_2 & (2,3) & (4,3)\\
\hline
x_3 & (4,3) & (2,3)\\
 \hline
  \end{array}\]

\centerline{Figure \arabic{figurecounter}: The payoff matrix in Example \ref{example:2}.}
\addtocounter{figurecounter}{1}
\bigskip

Let $0<\alpha'<\alpha<\frac{1}{2}$ and define $A:= B((3,3),\alpha')$ and
$D:= ([1-\alpha,3+\alpha] \times [1-\alpha,1+\alpha]\cup [3-\alpha,3+\alpha] \times [1-\alpha,3+\alpha])$, see Figure~\arabic{figurecounter}.

\begin{figure}[!h]
\label{fi:ex2}
\begin{center}
\begin{tikzpicture}
[scale=0.8,
    axis/.style={very thick, ->, >=stealth'},
    ]
\draw[axis] (0.5,0.4)  -- (0.5,6) ;
\draw[axis] (0.4,0.5) -- (8,0.5) ;


\draw (1.5,1.5) circle[radius=2pt];
\draw (4.5,1.5) circle[radius=2pt];
\draw (2.5,3.5) circle[radius=2pt];
\draw (4.5,3.5) circle[radius=2pt];
\draw (3.5,3.5) circle[radius=10pt];
\node[] at (3.5,3.5){$A$};
\draw[-] (3,4)--(3,2)--(1,2)--(1,1)  -- (4,1) -- (4,4) -- (3,4);
\draw[-] (1.5,0.4)--(1.5,0.6) ;
\node[] at (1.5,0){1};
\draw[-] (2.5,0.4)--(2.5,0.6) ;
\draw[-] (3.5,0.4)--(3.5,0.6) ;
\node[] at (3.5,0){3};
\draw[-] (4.5,0.4)--(4.5,0.6) ;
\draw[-] (5.5,0.4)--(5.5,0.6) ;
\node[] at (5.5,0){5};
\draw[-] (6.5,0.4)--(6.5,0.6) ;
\draw[-] (0.4,1.5)--(0.6,1.5) ;
\draw[-] (0.4,2.5)--(0.6,2.5) ;
\draw[-] (0.4,3.5)--(0.6,3.5) ;
\draw[-] (0.4,4.5)--(0.6,4.5) ;
\draw[-] (0.4,5.5)--(0.6,5.5) ;
\node[] at (0,1.5){1};
\node[] at (0,3.5){3};
\node[] at (0,5.5){5};
\end{tikzpicture}
\end{center}
\centerline{Figure \arabic{figurecounter}: The sets $A$ and $D$ in Example \ref{example:2}.}
\addtocounter{figurecounter}{1}
\end{figure}

To approach the set $A$ while remaining in $D$ with high probability, Player~1 can use the following strategy.
\begin{itemize}
\item
Play the action $x_0$ during $T_0$ stages, where $T_0 \in \dN$ is sufficiently large.
Regardless of the play of Player~2, the average payoff is $g_t = (1,1)$ for every $t \in \{1,2,\ldots,T_0\}$.
\item
Between stages $T_0+1$ and $T_0+T_1$, play the action $x_1$, where  $T_1=3T_0$
Regardless of the play of Player~2,
we have $g_{T_0+T_1} = (3,1)$, and
for every $t \in \{T_0+1,T_0+2,\ldots,T_0+T_1\}$
the average payoffs $g_t$ is in the convex hull of $(1,1)$ and $(3,1)$, hence in $D$.
\item
Then play forever the mixed action $[\frac{1}{2}(x_2),\frac{1}{2}(x_3)]$, which approaches the set $A$.
\end{itemize}
By the strong law of large numbers, the probability that the average payoff always remains in $D$ goes to 1 as $T_0$ goes to infinity.
Indeed, the average payoff will leave the set $D$ only if in the last phase of the play,
the percentage of the number of stages in which the outcome is $(2,3)$ is far from 0.5, an event which occurs with a small probability
when $T_0$ is sufficiently large.
\end{example}

While the strategy used in the proof of Theorem \ref{thm:main_thm} plays either the safe action or a strategy that approaches
the set $\overline{A} \cap \overline{D}$, the strategy used in Example \ref{example:2} starts by playing a
sequence of actions that lead the average payoff towards the intermediate target $(3,1)$, and then a sequence of actions that lead
the average payoff towards $\overline{A} \cap \overline{D}$. There are two main differences between the two strategies.

\begin{itemize}
\item First, since in Theorem \ref{thm:main_thm} the set $D$ is convex, the convex hull of $R_1(s)$ and any point $g_t \in D$ is a subset of $D$,
hence Player~1 can switch from playing the safe action $s$ to a strategy that approaches $\overline{A} \cap \overline{D}$, back and forth,
to maintain the average payoff in $D$. When the set $D$ is not convex, one needs to ``lead'' the average payoff from $R_1(s)$ to some point $x$
that satisfies that the convex hull of $x$ and  a subset of $\overline{A} \cap \overline{D}$ remains in $D$, and once the average payoff gets close to $x$,
switch to the strategy that approaches $\overline{A} \cap \overline{D}$. Thus, the play before switching to the strategy that approaches $\overline{A} \cap \overline{D}$ is more involved when the set $D$ is not convex.

\item Second, the convexity of the set $D$ in Theorem \ref{thm:main_thm} ensures that whenever the average payoff gets close to $F \setminus D$ and the strategy plays the safe action again,
the average payoff, which is the average of two vectors in $D$, is in $D$. When the set $D$ is not convex, this property does not hold, and it may be impossible to play the safe action,
in a way that ensures that the average payoff remains in $D$.
Hence in this case we have to work with the weaker concept of Definition~\ref{def:weak_approch_constr}.
\end{itemize}

The discussion above leads us to the following sufficient condition for a pair of set $(A,D)$ to be such that Player~1 can approach $A$ while remaining in $D$ with high probability.
A mixed action of Player~1 is called \emph{safe} if each action in the support of the mixed action is safe.

\begin{theorem}\label{thm:path-approach}
Let $A$ and $D$ be two subsets of $\mathbb{R}^n$, the latter being open. Suppose that
\begin{enumerate}
\item[(D1)] The set $A \cap D$ is approachable.
\item[(D2)] There exist $\delta > 0$, a safe mixed-action $x_0$,
$m$ mixed actions $x_1,\dots,x_m \in \Delta(I)$, and $m$ open subsets $A_1,\dots,A_m$ of $\mathbb{R}^n$ such that
$A_m \subseteq A  \cap D $.
Set $A_0 := R_1(x_0)$,
and assume that for every $0 \leq \ell \leq m-1$ the following hold:
\begin{eqnarray}
\label{equ:c2a}
&&\Conv[\overline{A}_{\ell} \cup  B(R_1(x_{\ell+1}),\delta)] \setminus A_{\ell+1} \hbox{ is not path-connected, and}\\
&&\Conv[A_{\ell} \cup B(A_{\ell+1},\delta)]  \subseteq D.
\label{equ:c2b}
\end{eqnarray}
\end{enumerate}
Then Player~$1$ can approach $A$ while remaining in $D$ with probability $1-\epsilon$, for every $\epsilon>0$.
\end{theorem}

\begin{proof}
The proof is quite technical, yet it poses no conceptual difficulties.
We therefore only present the main steps of the proof.

Eq.~(\ref{equ:c2a}) implies that every path in $\Conv[A_{\ell} \cup  R_1(x_{\ell+1})]$ that links $A_{\ell}$ to $R_1(x_{\ell+1})$ intersects $A_{\ell+1}$.
Moreover, because $A_{\ell+1}$ is open while $\overline{A}_\ell$ and $R_1(x_{\ell+1})$ are closed,
the length of this intersection is bounded away from 0.
Denote by $\Lambda_\ell > 0$ a lower bound on the length of these intersections and define
$$\Lambda:= \displaystyle \min_{0\leq \ell \leq m-1} \Lambda_{\ell}>0.$$

Fix $\ep > 0$ and let $T$ such that $T \gg \frac{1}{\Lambda}$.
Define a collection $\tau_0,\tau_1,\cdots,\tau_{m-1}$ of stopping times as follows:
\begin{eqnarray}
\tau_0 &:=& T,\\
\tau_{\ell} &:=& \min\{t > \tau_{\ell-1} \colon g_t \in A_\ell\}, \ \ \ \ell \in \{1,2,\cdots,m-1\}.
\end{eqnarray}
Define a strategy $\sigma(T)$ as follows:
\begin{itemize}
\item   Until stage $\tau_0$ play the safe mixed-action $x_0$.
\item   Between stages $\tau_{\ell-1}$ and $\tau_\ell-1$ play the mixed action $x_\ell$, for each $\ell \in \{1,2,\cdots,m-1\}$.
\item   From stage $\tau_{m-1}$ and onward play a strategy $\widehat\sigma$ that approaches the set $A \cap D$.
\end{itemize}
Let $T_0$ be sufficiently large.
We argue that if Player~1 plays the strategy $\sigma(T_0)$,
then with high probability
\begin{itemize}
\item
The stopping times $\tau_1,\cdots,\tau_{m-1}$ are bounded, regardless of the strategy played by Player~2.
\item
The average payoff remains in $D$.
\end{itemize}
Indeed, by the strong law of large numbers, provided $T_0$ is sufficiently large, we have $P_{\sigma(T_0),\tau}[\tau_{0} \leq T_{0}] > 1-\ep$,
for every strategy $\tau$ of Player~2.
Moreover, since $x_0$ is a safe mixed action, $g_t \in D$ for every $t \leq \tau_0$.

Assume by induction that, for a given $\ell \in \{1,2,\cdots,m-2\}$, there is an integer $T_{\ell-1}$ such that $P_{\sigma(T),\tau}[\tau_{\ell-1} \leq T_{\ell-1}] > 1-\ell\ep$.
In particular, $g_{\tau_{\ell-1}} \in A_{\ell-1}$ on the event $\{\tau_{\ell-1} \leq T_{\ell-1}\}$.
At stage $\tau_{\ell-1}$ Player~1 starts playing the mixed action $x_\ell$,
so the expected stage payoff is in $R_1(x_\ell)$.
The sequence of the average payoffs $(g_t)_{t=\tau_{\ell-1}}^{\tau_\ell}$ starts at $A_{\ell-1}$ and moves towards $R_1(x_\ell)$.
Since $\tau_{\ell-1} \geq \tau_0 = T_0$, we have $\|g_t-g_{t-1}\| < \tfrac{2}{T_0} < \Lambda$.
By the strong law of large numbers, provided $T_0$ is much larger than $\tfrac{1}{\Lambda}$,
there is $T_\ell \geq T_{\ell-1}$ such that $g_t \in A_{\ell}$ for some $t \leq T_\ell$, and therefore $\tau_\ell$ is finite.
Moreover, Eq.~\eqref{equ:c2b} implies that provided $T_0$ is sufficiently large, $P_{\sigma(T_0),\tau}[\tau_{\ell} \leq T_{\ell}] > 1-(\ell+1)\ep$,
for every strategy $\tau$ of Player~2.

Since at stage $\tau_{m-1}$ Player~1 starts following a strategy that approaches the set $A \cap D$,
there is an integer $T_m$ such that $P_{\sigma(T_0),\tau}[d(g_t, A \cap D) \leq \ep, \ \ \ \forall t \geq \tau_m] > 1-(m+2)\ep$.
Eq.~\eqref{equ:c2b}, applied to $\ell=m$, implies that, provided $T_0$ is sufficiently large, with high probability $g_t \in D$ for every $t \geq \tau_m$.
The result follows.
\end{proof}

The next example shows that the sufficient condition provided by Theorem~\ref{thm:path-approach} is not necessary.

\begin{example}
\label{example3}
Consider the game that appears in Figure~\arabic{figurecounter},
where Player~1 has three actions, $T_1$, $T_2$, and $B$, Player~2 has two actions, $L$ and $R$,
and the payoff is two-dimensional.

\[ \begin{array}{c|c|c|}
 & L & R \\
\hline
T_1 & (1,2) & (1,2) \\
\hline
T_2 &  (2,1) & (2,1)\\
\hline
B &  (2,3) & (3,2)\\
\hline
 \end{array}\]

\centerline{Figure \arabic{figurecounter}: The payoff matrix in Example \ref{example3}.}
\addtocounter{figurecounter}{1}
\bigskip

Let $0<\alpha'<\alpha<\frac{1}{2}$ and define $A:= \{(2,2)\}$ and
$D:= ([2-\alpha,2+\alpha] \times [2-\alpha,3+\alpha] \cup [2-\alpha,3+\alpha] \times [2-\alpha,2+\alpha])$,
see Figure~\arabic{figurecounter}.

\begin{figure}[h!]
\begin{center}
\begin{tikzpicture}
[scale=0.6,
    axis/.style={very thick, ->, >=stealth'},
    dashed line/.style={dashed, thin},
    pile/.style={thick, ->, >=stealth', shorten <=2pt, shorten
    >=2pt},
    every node/.style={color=black}
    ]

\draw[axis] (0.5,0.4)  -- (0.5,12) ;
\draw[axis] (0.4,0.5) -- (12,0.5) ;

\draw (2,5) circle[radius=2pt];
\draw (8,5) circle[radius=2pt];
\node[] at (5,5){$A$};
\draw (5,5) circle[radius=10pt];

\draw (5,2) circle[radius=2pt];
\draw (5,8) circle[radius=2pt];

\draw[-] (4,4)  -- (4,9);
\draw[-] (4,4)  -- (9,4);

\draw[-] (6,6)  -- (6,9);
\draw[-] (6,6)  -- (9,6);

\draw[-] (4,9)  -- (6,9);
\draw[-] (9,4)  -- (9,6);

\draw[-] (2,0.4)--(2,0.6) ;
\node[] at (2,0){1};

\draw[-] (5,0.4)--(5,0.6) ;
\node[] at (5,0){2};

\draw[-] (8,0.4)--(8,0.6) ;
\node[] at (8,0){3};

\draw[-] (0.4,2)--(0.6,2) ;
\node[] at (0,2){1};

\draw[-] (0.4,5)--(0.6,5) ;
\node[] at (0,5){2};

\draw[-] (0.4,8)--(0.6,8) ;
\node[] at (0,8){3};

\end{tikzpicture}

Figure \arabic{figurecounter}: The sets $A$ and $D$ in Example \ref{example3}.
\end{center}
\end{figure}

\addtocounter{figurecounter}{1}

\bigskip

In this game, $B$ is the only safe action of Player~1, yet $R_1(B) \not\subset D$,
and therefore Eq.~(\ref{equ:c2b}) is not satisfied for $\ell=0$.
Nevertheless, Player~1 can approach $A$ while remaining in $D$.
To do that, Player~1 plays in blocks of size 2.
In the first stage of each block Player~1 plays the action $B$.
If Player~2 played the action $L$ (resp. $R$) at the first stage of the block, then in the second stage of the block Player~1 plays the action $T_2$ (resp. $T_1$).
The average payoff in each block is $(2,2)$.
The reader can verify that the strategy described above ensures that the average payoff converges to $(2,2)$ while remaining in $D$.
\end{example}

\section{Open Problems}
\label{section:open}

We end the paper by mentioning two open problems.
Theorem~\ref{thm:path-approach} provides a sufficient condition for approachability with constraints in the nonconvex case.
A natural question is whether one can find a necessary and sufficient condition for approachability with constraints in this case.
As the examples shows, a complete characterization may require handling many different cases and be quite technical.

Another issue, that already arises in the setup of approachability without constraints,
is what happens when Player~2 (nature) cannot play any mixed action, but is restricted to a subset of mixed actions.
This happens, e.g., when Player~2 is composed of two players who cannot correlate their actions.
Can one find out how the collection of sets that are approachable when Player~2 changes as the set of mixed action available to Player~2 varies?

\bibliographystyle{plainnat}
\bibliography{bib_approach}

\end{document}